\documentclass{article}
\usepackage[longnamesfirst]{natbib}
\bibpunct[ ]{(}{)}{,}{a}{}{,}
\usepackage{graphicx}
\usepackage{amssymb}
\usepackage{amsmath}
\usepackage{amstext}
\usepackage{amsfonts}
\usepackage{amsthm}
\usepackage{bm}
\usepackage{mathtools}
\author{Daniel Spivak}

\usepackage{hyperref}
\hypersetup{
colorlinks=true, linkcolor=black, %blue,
citecolor=OliveGreen
}
\usepackage{comment}
\usepackage{xcolor,graphics}
\usepackage[autostyle, english = american]{csquotes}
\MakeOuterQuote{"}

\definecolor{OliveGreen}{rgb}{0.1, 0.4, 0.1}
\definecolor{gray-asparagus}{rgb}{0.0, 0.6, 0.3}

\title{One-point large deviations of the directed landscape geodesic}

\theoremstyle{plain}
\newtheorem{thm}{Theorem}[section]

\newtheorem{pro}[thm]{Proposition}
\newtheorem{lem}[thm]{Lemma}

\theoremstyle{definition}

\newtheorem{prob}[thm]{Problem}

\numberwithin{equation}{section}

\begin{document}

\maketitle

\begin{abstract}

The directed landscape, the central object in the Kardar-Parisi-Zhang universality class, is shown to be the scaling limit of various models by \cite{dauvergne2022scalinglimitlongestincreasing} and \cite{Dauvergne_Ortmann_Virag_2018}. In his study of geodesics in upper tail deviations of the directed landscape, \cite{Zhipeng_Liu_2022a} put forward a conjecture about the rate of the lowest rate metric under which a geodesic between two points passes through a particular point between them. \cite{Das_Dauvergne_Virag_2024} disproved his conjecture, and made a conjecture of their own. This paper disproves that conjecture and puts the question to rest with an answer and a proof.
\end{abstract}

\section{Introduction}

The directed geodesic $\gamma$ is a random continuous function $[0,1]\to \mathbb R$. It is defined as the unique geodesic from $(0,0)$ to $(0,1)$ in the directed landscape and is the scaling limit of geodesics of several last-passage percolation models.  Our main result answers the conjectures of\
%Liu in 
\cite{Zhipeng_Liu_2022a} and 
%Das, Dauvergne and Virag in
\cite*{Das_Dauvergne_Virag_2024}.

\begin{thm} \label{ratetheorem} For any fixed $t\in (0,1)$ as $r\to\infty$ we have 
$$ P(\gamma(t)\ge r)=\exp\left\{-r^3\frac{8+o(1)}{3a^2(3-\sqrt{8a})^2}\right\}, \qquad a=\min(t,1-t). 
$$\end{thm}

We can also give the precise limit shape of the geodesic under this conditioning. 
Let $\gamma_{a,r}$ be a geodesic sampled from the conditional law of $\gamma$ given that $\gamma(a)\ge r$.

\begin{thm} \label{curveshapetheorem}
A $r \rightarrow \infty$,  $\gamma_{a,r}/r\to F_a$ in probability with respect to uniform convergence, where
$$ F_a(t)=
    \begin{cases}
      \frac ta, & \text{if } t \leq a \\
      1 + \frac{t-a}{a}(\frac{4(1-\sqrt {2a})}{3-\sqrt{8a}}-1), & \text{if } a \leq t \leq 2a
      \\
      \frac{4(\sqrt{t}-t)}{\sqrt{2a}(3-\sqrt{8a})}, & \text{if } 2a \leq t \leq 1.
    \end{cases}
$$
for $a\le 1/2$, and $F_a(t)=F_{1-a}(1-t)$ for $a\ge 1/2$. 

\end{thm}

The limiting shape of these geodesics is shown in the following two figures. Paradoxically, for $a < \frac 18$, the maximum of the geodesic is not attained at $a$, and the geodesic moves away from its final destination until $t = \frac 14$. The following two graphs show these geodesics for some values of $a$.

\begin{figure}[h]
		\centering
		\includegraphics[scale=0.6]{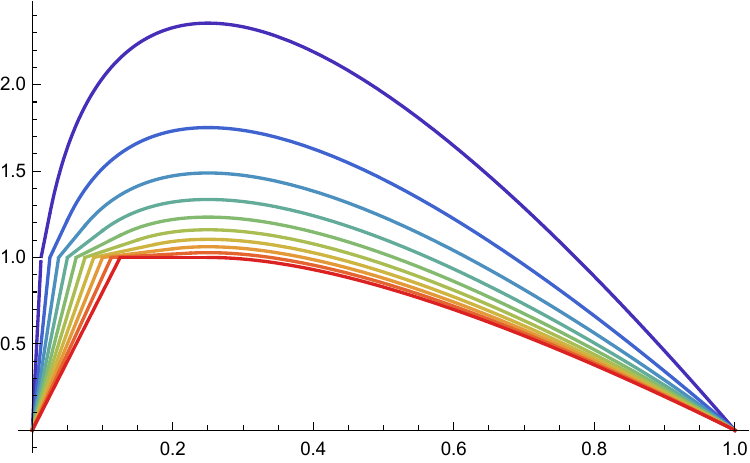}
		\caption{$0 < a \leq \frac 18$}
	\end{figure}

	\begin{figure}[h]
		\centering
		\includegraphics[scale=0.6]{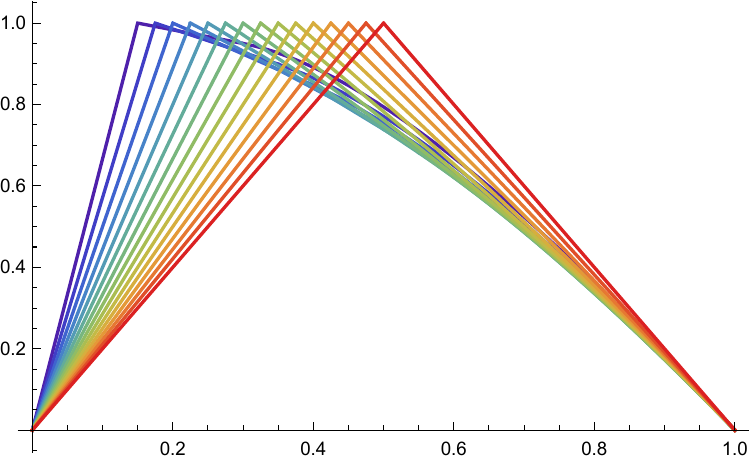}
		\caption{$\frac 18 < a \leq \frac 12$}
	\end{figure}

The directed geodesic has been studied recently, starting with the work in 
\cite{Dauvergne_Ortmann_Virag_2018} showing the existence and uniqueness of this random continuous function for a given start and end point almost surely. \cite{Dauvergne_Sarkar_Virag_2020} proved that the directed geodesic has $\frac 32$-variation, and \cite{Zhipeng_Liu_2022b} found the scale of its fluctuations conditioned on the geodesic being unusually long. The recent work by \cite{ganguly2023brownianbridgelimitpath}  shows that after rescaling, the long directed geodesic converges to a Brownian bridge. \cite{Zhipeng_Liu_2022a} gives a formula for the joint distribution of the last passage geodesic passing through a point, its length before the point and its length after the point, but it is not clear how one can deduce asymptotic tail probabilities from this formula.

The special case when $a=\frac 12$ was solved in \cite{agarwal2024sharpdeviationboundsmidpoint}, and for this value of $a$ the geodesic made of two straight line is optimal (compare this to \ref{curveshapetheorem}).

Let $\mathbb R_{\uparrow}^4$ denote pairs of elements of spacetime $(x,s,y,t)$ satisfying $s<t$. A directed metric (or metric) is a continuous function $e : \mathbb R_{\uparrow}^4 \rightarrow \mathbb R$ satisfying the (reverse) triangle inequality 

\begin{equation} \label{introtriangle}
    e(x,s,y,t) + e(y,t,z,u) \leq e(x,s,z,u)
\end{equation}

The directed landscape $\mathcal L$ is a random directed metric that are defined by:

\begin{enumerate}
    \item (Airy sheet marginals) For any $s<t$
    
    \begin{equation}
        \mathcal L(x,s,y,t) \stackrel{d}= \mathcal S(x(t-s)^{-\frac 23},y(t-s)^{-\frac 23})
    \end{equation}
    jointly in all $x,y$.
    \item
        (Independent increments) For any disjoint set of time intervals $\{ [s_i,t_i] \}$, the random functions $\mathcal L(\cdot,s,\cdot,t)$ are independent.
    \item (Metric composition law) For $r<s<t$ and for any $x,y$,
    \begin{equation}
        \mathcal L(x,r,y,t) = \max_z (\mathcal L(x,r,z,s) + \mathcal (z,s,y,t))
    \end{equation}
    \end{enumerate} 

\cite{Dauvergne_2024} showed that, almost surely, every pair of points is connected by one of 27 isomorphism types of geodesic networks.

\cite*{Das_Dauvergne_Virag_2024} showed that every pair of points are connected by a unique rightmost geodesic. They study the rescaling 

\begin{equation}
    \mathcal L_{\varepsilon}(x,s,y,t) = \varepsilon \mathcal L(\varepsilon^{-\frac 12}x,s,\varepsilon^{-\frac 12}y,t)
\end{equation}

They show the following Large deviation principle. For any open set $U$ and any closed set $V$ in the topology of uniform convergence on bounded sets,

\begin{equation}
\begin{aligned}
    \liminf_{\varepsilon \rightarrow 0} -\varepsilon^{\frac 32}\log{P(\mathcal L_{\varepsilon} \in U)} \leq \inf_{e \in U} I(e) \\
    \limsup_{\varepsilon \rightarrow 0} -\varepsilon^{\frac 32}\log{P(\mathcal L_{\varepsilon} \in V)} \geq \inf_{e \in V} I(e)
\end{aligned}
\end{equation}
where $I(e)$ is a lower semicontinuous function called the rate function. Moreover, $I$ is a "good rate function", meaning that $I^{-1}[0,r]$ is compact for any $r \geq 0$, and has the following concrete description:

The Dirichlet metric $d$ is given by $d(x,s,y,t) = -\frac{(y-x)^2}{t-s}$. All finite rate metrics $e$ satisfy $e \geq d$, and arise from "planting measures" along countably many curves: For a curve $\gamma$ and a nonnegative measure density $\rho$, we can define
\begin{equation}
    e_{\gamma,\rho}(\gamma(t_1),t_1,\gamma(t_2),t_2) = \int_{t_1}^{t_2} (\rho(t) - \gamma'(t)^2) dt
\end{equation}
where $e_{\gamma, \rho}(x_1,t_1,x_2,t_2) = - \infty$ for any other choice of $x_1,x_2$. Then the metric $e$ arising from planting measures along a finite or countable set of curves is the smallest metric which is bounded below by satisfying $e \geq d$, and for each curve $\gamma_i$ and planted measure $\rho_i$, $e \geq e_{\gamma_i, \rho_i}$. All metrics not of this form have infinite rate.

For each such planted measure, the contribution to the rate function is $I_i = \frac 43 \int \rho_i^{\frac 32} dt$, so that the rate function $I$ of a metric $e$ is the sum over all planted measures $I = \sum_i I_i$.
Our goal in this paper is to find the lowest rate metric for which a geodesic $F$ from $(0,0)$ to $(0,1)$ passes through the point $(1,a)$.

Our metrics will generally have a measure planted along just one curve, which will be the geodesic from $(0,0)$ to $(0,1)$. Any measure planted outside of this geodesic will not increase its length, but may increase the length of paths other than $F$, which only reduces the possibility of $F$ being a geodesic, all while increasing the rate.

We can then define $e$ or $e(F,\rho)$ as the metric arising from planting $\rho$ along $F$, which will be our geodesic.

We will in fact prove a result about $e$: 
Suppose $\mathcal M_{\varepsilon}$ is a metric sampled from the conditional law of $\mathcal L_\varepsilon$, given that the rightmost geodesic $\gamma$ of $e$ from $(0,0)$ to $(0,1)$ satisfies $\gamma(a) \geq 1$.

\begin{thm} \label{metrictheorem}
    As $\varepsilon\to 0$, $\mathcal M_{\varepsilon} \rightarrow e(F,\rho)$ in probability with respect to uniform convergence on bounded sets, with $F$ as in Theorem \ref{curveshapetheorem} and $\rho$ as given by:
    $$ \rho(t)=
    \begin{cases}
      \frac{1}{a^2(3-\sqrt{8a})^2}, & \text{if } t \leq 2a \\
      \frac{2}{at(3-\sqrt{8a})^2}, & \text{if } 2a \leq t \leq 1.
    \end{cases}
$$

\end{thm}

\section{Solution of a relaxed problem}

One possible candidate for $F$ is two straight line segments, from $(0,0)$ to $(1,a)$ and from $(1,a)$ to $(0,1)$, with $\rho$ just sufficient to compensate $d$, such that $e(0,0,F(t),t) = 0, \forall t \in [0,1]$. However, for $a < \frac 12$, this will not be a geodesic: instead, there will be straight lines from $(0,0)$ to a point beyond $(1,a)$ whose Dirichlet length is greater than the length of $F$ to that point. As a result, the geodesic will skip the steeper line segment to $(1,a)$ and take such a straight line as a shortcut to $F$, and then follow part of the other line segment to $(0,1)$. Similarly, if $a > \frac 12$, it will follow part of the line segment from $(0,0)$ to $(1,a)$ and then shortcut the steep segment to proceed directly to $(0,1)$.

This suggests that we should try to solve for a metric where the segment from $(0,0)$ to $(1,a)$ is not skipped by such a shortcut. In other words, we should ask what is the lowest rate metric for which there is a geodesic candidate $F$, so that the distance from $(0,0)$ along $F$ to any point on $F$ is greater than the distance of a shortcut from $(0,0)$ to that point.

Inspired by this intuition, we will first solve an intermediate optimization problem: fix $F_0$ and $\rho_0$ (and therefore $e(F_0,\rho_0)$) for $t \in [0,t_1]$. Consider pairs $(F,\rho)$ that agree with $(F_0,\rho_0)$ on $[0,t_1]$. Which pair minimizes the rate of the metric $e=e(F,\rho)$ on the time interval $[t_1,t_2]$, with the condition that the path along $F$ up to any $t\in [t_1,t_2]$ is not worse than the shortcut from $(0,0)$ to $(F(t),t)$? In other words, what is the lowest rate metric $e=e(F,\rho)$ on $[t_1,t_2]$  such that  

\begin{equation} \label{d=e+e}
    d(0,0,F(t),t) \leq  e(0,0,x_1,t_1) + e(x_1,t_1,F(t),t), \forall t \in (t_1,t_2) 
\end{equation}
In theory, the parameters of this optimization problem given by $e$ and $F$ up to time $t_1$. However, the only information used by the optimization problem are the scalar $e(0,0,x_1,t_1)$ and the starting point $x_1=F_0(t_1).$

We will store 
\begin{equation} \label{eq23}
b=    e(F_0,\rho_0)(0,0,x_1,t_1) -d(0,0,x_1,t_1)= e(0,0,x_1,t_1) - \frac{x_1^2}{t_1}
\end{equation}
as an initial condition in the form of the constant $b$ (where $b \geq 0$), so that we can focus on optimizing the path on $(t_1,t_2)$, and assume that $b$ is not so big as to cause \eqref{d=e+e} to hold for arbitrary $e \geq d$.

We will also assume that $\frac{x_1}{t_1} > \frac{x_2}{t_2}$. The problem is identical for $\frac{x_1}{t_1} < \frac{x_2}{t_2}$, but in this case, the entire picture should be flipped left to right.

We will define the function $F$, the rightmost geodesic candidate from $(x_1,t_1)$ to $(x_2,t_2)$, and $\rho$, the measure density planted along the path $(F(t),t)$, with respect to the Lebesgue measure on $t$.
Since we are optimizing over $t \in [t_1,t_2]$, we consider only the contribution to the rate function of the measure planted within this interval, $I_{2}$. Now we will formally state the relaxed optimization problem that we will solve. 

\begin{prob} \label{relaxedstatement}
    Given real numbers $t_1,x_1,t_2,x_2,b$ satisfying
    \begin{itemize}
    \item $0 < t_1 < t_2$
    \item $\frac{x_1}{t_1} > \frac{x_2}{t_2}$
    \item $0 \leq b \leq \frac{(x_2-x_1)^2}{t_2-t_1} + \frac{x_1^2}{t_1} - \frac{x_2^2}{t_2}$.
    \end{itemize}
    
    Minimize   
\begin{equation}
    I_{2}(F,\rho) := \frac 43 \int_{t_1}^{t_2}\rho^{\frac 32} dt
\end{equation}
    over $F$, $\rho$ satisfying
    \begin{itemize}
    \item $F : [t_1,t_2] \rightarrow \mathbb R$ is an absolutely continuous function, satisfying $F(t_1) = x_1$, $F(t_2) = x_2$.
    \item $\rho : [t_1,t_2] \rightarrow \mathbb R$ is a nonnegative Lebesgue-measurable function.
    \item $g:[t_1,t_2] \rightarrow \mathbb R$,
    \begin{equation}\label{defineg}
        g(t) := (b - \frac{F(t_1)^2}{t_1}) + \int_{t_1}^{t} (\rho(\tau)-F'(\tau)^2)d\tau + \frac{F(t)^2}{t} \geq 0, \forall t \in [t_1,t_2].
    \end{equation}
    \end{itemize}
\end{prob}

We will call a pair $(F,\rho)$ feasible if it satisfies all of the above conditions (without necessarily minimizing $I_{2}$). 
The answer to this problem is given by the following proposition.
\begin{pro} \label{bigtheorem}
    There exists a unique minimizing feasible pair $(F,\rho)$ for $I_{2}$, up to almost everywhere equivalence of $\rho$. Such a pair must satisfy the following conditions:
    
    There exists $t_{B} \in [t_1,t_2]$ such that $F$ is linear and $\rho$ is constant on $[t_1,t_{B})$, and $F(t) = c_1 \sqrt t + c_2 t$ and $\rho(t) = (\frac{F(t)}s-F'(t))^2$ for $t \in (t_{B},t_2]$, with $t_{B}=t_1$ if and only if $b=0$.   
    
    Moreover, $g', F', \rho$ are continuous at the merge point $t_{B}$.
\end{pro}

The function $g$ keeps track of the difference in $\eqref{d=e+e}$. 
    We can also rewrite \eqref{defineg} as 
    \begin{equation} \label{gprime}
    g(t) = b + \int_{t_1}^{t} (\rho(\tau) - (\frac{F(\tau)}{\tau} - F'(\tau))^2) d\tau
    \end{equation}
    and so $g'(t) = \rho(t) - (\frac{F(t)}t - F'(t))^2$ almost everywhere.

    The integral $\int_{t_1}^t \rho(\tau) d\tau$ is the contribution coming from the density $\rho$ planted along the curve. The remainder of the integral, $\int_{t_1}^t -F'(\tau)^2 d\tau$ is the Dirichlet metric term, penalizing $F$ for excessive curvature. The other terms, $\frac{F(t)^2}t - \frac{F(t_1)^2}{t_1} = d(0,0,F(t),t) - d(0,0,F(t_1),t_1)$, represent the difference in the Dirichlet metric of length of the shortcut to $t$, as compared to that to $t_1$. 

    Lastly, $b = g(t_1)$ should be interpreted as the result of having previously optimized over $t < t_1$.

    We may also write $g(F,\rho,b)$ instead of $g$ to emphasize which metric this $g$ arises from.

 When it is not ambiguous, we may write $I_{2}$ instead of $I_{2}(F,\rho)$. 
Similarly, we will write $e$ for the associated metric $e(F,\rho)$. 

We will prove Proposition \ref{bigtheorem} by establishing some basic facts about the minimizer, thereby reducing to two tractable optimization problems, one when $g$ is zero and one when $g$ is nonzero. We start with some generic lemmas about absolutely continuous functions.

\begin{lem}\label{derivativelembad}
    Suppose $F$ is absolutely continuous on an interval $J \subset [t_1,t_2]$. Then $\frac{F(t)}t$ is nonincreasing on $J$ if and only if $F'(t) \leq \frac{F(t)}t$ for $t \in J$, and is constant on $J$ if and only if $F'(t) = \frac{F(t)}t$ on $J$.
\end{lem}
\begin{proof}
$\frac{F(t)}t$ is absolutely continuous if and only if $F$ is, and 
    \begin{equation}
        \frac d{dt} \frac{F(t)}t = \frac 1t (F'(t) - \frac{F(t)}t).
    \end{equation}
\end{proof}

\begin{lem}\label{maximumlemma}
Suppose $F(t) = \int_a^t f(\tau) d\tau$ is an absolutely continuous function on $[A,B]$, and $G(t) = max_{\tau \leq t} F(\tau)$. Then $G(t) = \int_a^t g(\tau) d\tau$, where $g = f \chi_{_{F=G}}$.
\end{lem}
\begin{proof}
    If $V = \int |f|$ is the variation of $F$, define $H=V-G$. If $x<y$, clearly $G(x) \leq G(y)$. If $G(y)>G(x)$, suppose that $z \in (x,y]$ with $F(z)=G(y)$. Then \begin{equation}
        V(y) \geq V(x) + (F(z)-F(x)) \geq V(x) + (G(y) - G(x)).
    \end{equation}
    It follows that $H$ is nondecreasing, so $G=\int g$ and $H = \int h$ with $g+h = |f|$ and $g,h \geq 0$. When $F(t) \neq G(t)$, we have $g=0$ almost surely. When $F(t)=G(t)$, we must have $F'(t) \geq 0$ (if $F'$ exists), so almost surely either $f=0$ or $f>0$ and $F$ is locally increasing. In both cases, $h=0$ almost surely. Therefore, $g = |f| \chi_{_{F=G}} = f \chi_{_{F=G}}$.
\end{proof}

\begin{lem}\label{dirichlet of majorant}
    Suppose $G$ is absolutely continuous on $[A,B]$, and
    \begin{equation}
        \hat F(\tau) = \sup_{s < \tau < t} (\frac{t-\tau}{t-s}G(s) + \frac{\tau-s}{t-s}G(t))
    \end{equation}

    is the concave majorant of $G$. Then $\int_A^B G'(\tau)^2 d\tau \geq \int_A^B \hat F(\tau)^2 d\tau$.
\end{lem}
\begin{proof}

    Recall that for a partition $P = \{a=\tau_1,\tau_2,\ldots,\tau_n = B \}$ of $[A,B]$, if
    \begin{equation} 
    d_{F}(P) = \sum_{i=1}^{n-1} \frac{F(\tau_{i+1})^2-F(\tau_i)^2}{\tau_{i+1}-\tau_i},
    \end{equation}
    then the Dirichlet length of $F$ is $\sup_P{d_{F}(P)}$; see Lemma 5.1.6 in \cite{dembo2009large}.

    Suppose $P$ is a partition of $[A,B]$. For $2 \leq i \leq n-1$, let $\hat F(\tau_i) = \frac{t_i-\tau_i}{t_i-s_i}G(s_i) + \frac{\tau-s_i}{t_i-s_i}G(t_i)$, where $\hat F$ agrees with $G$ on $s_i, t_i$ (if $\hat F(\tau_i) = G(\tau_i)$, set $s_i=t_i=\tau_i$). Let $P_-$ be the partition of $[t_1,t]$ by including $s_i$ and $t_i$, and $P_+ = P_- \cup P$. Since $P_+$ is a refinement of $P$, $d_{\hat F}(P) \leq d_{\hat F}(P_+)$. Since $\hat F$ is a straight line between $s_i$, $\tau_i$ and $t_i$, $d_{\hat F}(P_+) = d_{\hat F}(P_-)$. Since $\hat F$ and $G$ agree on $P_-$, we have $d_{\hat F}(P_-) = d_{G}(P_-)$, so taking supremums we get that
    \begin{equation}
        \int_{t_1}^t G'(\tau)^2  d\tau \geq \int_{t_1}^t  \hat F'(\tau)^2  d\tau.
    \end{equation}
\end{proof}

\begin{lem} \label{concave}
    If $(F,\rho)$ is feasible, then there exists a feasible pair $(\hat F, \rho)$ where $\hat F$ is concave and satisfies $\hat F'(t) \leq \frac{\hat F(t)}t$ for all $t \in [t_1,t_2]$. In particular, $I_{2}(F,\rho) = I_{2}(\hat F, \rho)$. Moreover, if $F \neq \hat F$, then $g(F,\rho,b)(t_2) < g(\hat F, \rho, b)(t_2)$.
\end{lem}

\begin{proof}
    Note that $F(t)$ is absolutely continuous if and only if $\frac{F(t)}t$ is. Define $\frac{G(t)}t = \min_{\tau \leq t} \frac{F(\tau)}{\tau}$. Then by Lemma \ref{maximumlemma}, $\frac{G(t)}t$ is absolutely continuous and is the integral of $(\frac{F(t)}{t})'\chi_{_A}$, where $A = \{ t : \frac{F(t)}t = \frac{G(t)}t \}$. It follows that $G$ is absolutely continuous and its derivative is almost everywhere equal to $F'(t)$ on $A$ and to $\frac{G(t)}t$ on $[t_1,t_2] \backslash A$.
    
    In both cases, $(\frac{G(t)}{t}-G'(t))^2 \leq (\frac{F(t)}{t}-F'(t))^2$, so for $t \in [t_1,t_2]$, we have

    \begin{equation}
        g(G,\rho,b)(t) = \int_{t_1}^t \rho - (\frac{G(t)}{t}-G'(t))^2 \geq \int_{t_1}^t \rho - (\frac{F(t)}{t}-F'(t))^2  = g(F,\rho,b),
    \end{equation}
    with equality at $t=t_2$ if and only if $G'=F'$ almost everywhere, i.e. if and only if $G=F$.

    Next, we will define $\hat F$ as the concave majorant of $G$, by \begin{equation}
        \hat F(\tau) = \sup_{s < \tau < t} (\frac{t-\tau}{t-s}G(s) + \frac{\tau-s}{t-s}G(t)).
    \end{equation}

    Note that either this supremum is attained or $\hat F(\tau) = G(\tau)$. Since $G$ was continuous, $\hat F$ is also continuous.
    
    Since it is concave, $\hat F$ is absolutely continuous. We will show that $\hat F$ inherits from $G$ the property that $\frac{\hat F(\tau)}{\tau} \geq \hat F'(\tau)$. When $\hat F \neq G$, we have $\hat F(\tau) = \frac{t-\tau}{t-s}G(s) + \frac{\tau-s}{t-s}G(t)$ and $\hat F'(\tau) = \frac{G(t)-G(s)}{t-s}$. By definition of $G$, $\frac{G(t)}t \leq \frac{G(s)}s$, so
    \begin{equation} \label{twopointnine}
        \hat F'(\tau) = \frac{G(t)-G(s)}{t-s} \leq \frac{G(t)}t \leq \frac{G(s)}s.
    \end{equation}
    since $\frac{G(t)}t$ is the mediant of the other two fractions. Then by \eqref{twopointnine} twice,
    \begin{equation}
        \frac{\hat F(\tau)}{\tau} = \frac{1}{\tau}(\frac{t-\tau}{t-s}G(s) + \frac{\tau-s}{t-s}G(t)) \geq \frac{\hat F'(\tau)}{\tau(t-s)} ((t-\tau)s + (\tau-s)t) = \hat F'(\tau).
    \end{equation}
    Next, If $\tau$ is a limit point of the set where $\hat F \neq G$, then with countably many exceptions at the jump discontinuities of $\hat F'$, we can infer the same inequality by taking limits. Otherwise, $\hat F$ is equal to $G$ in a neighborhood of $\tau$, so $\hat F'(\tau) = G'(\tau) \leq \frac{G(t)}t = \frac{\hat F(t)}t$.

    To see that $\hat F$ is feasible, we will compare $g(\hat F, \rho, b)$ and $g(G,\rho,b)$. First consider $t \in [t_1,t_2]$ with $\hat F(t)=G(t)$. By Lemma \ref{dirichlet of majorant},
    \begin{equation}
        \int_{t_1}^t -\hat F'(\tau)^2  d\tau \geq \int_{t_1}^t - G'(\tau)^2  d\tau.
    \end{equation}
    Thus $g(\hat F, \rho, b)(t) \geq g(G, \rho, b)(t)$, with equality if and only if $G$ is linear on every interval where $\hat F$ is linear, i.e. if $\hat F$ coincides with $G$ on $[t_1,t]$. In particular, if $\hat F \neq G$, then $g(\hat F, \rho, b)(t_2) > g(G, \rho, b)(t_2)$.

    Now let $\tau \in [t_1,t_2]$ with $\hat F(\tau) \neq G(\tau)$, so that $\hat F(\tau) = \frac{t-\tau}{t-s}G(s) + \frac{\tau-s}{t-s}G(t)$. Let $\hat G$ be equal to $G$ except on $[s,\tau]$, where it is replaced by a straight line from $(G(s),s)$ to $(G(\tau), \tau)$. Then clearly $g(\hat G, \rho, b)(\tau) \geq g(G, \rho, b)(\tau)$.
    Observe that for a linear function $L$, the integral $\int_s^{\tau} -(\frac{L(t')}{t'}-L'(t'))^2$ in \eqref{gprime} is a function of $L(s)$ and $L'$, quadratic in $L'$, and maximized at $L'(s) = \frac{L(s)}s$. Then since $\hat G' \leq \hat F' \leq \frac{F(s)}s$, we have $g(\hat F, \rho, b)(\tau) \geq g(\hat G, \rho, b)(\tau)$.
    
    As a result, $(\hat F, \rho)$ is feasible, and $\hat F$ is concave with $\hat F'(t) \leq \frac{\hat F(t)}{t}$ for all $t \in [t_1,t_2]$.
    
    When $F \neq \hat F$, either $G \neq F$ or $\hat F \neq G$, and in both cases we have shown that $g(F,\rho,b)(t_2) < g(\hat F, \rho, b)(t_2)$.

    Lastly, $I(F,\rho) = I(\hat F, \rho)$, since $I$ depends only on $\rho$.
        
\end{proof}

With the added control on $F$, the set of geodesics becomes compact, if their rate does not diverge to infinity:

\begin{lem} \label{uniformF}
    Suppose $F_i$ are concave and satisfy $F_i'(t) \leq \frac{F_i(t)}t$ as in Lemma \ref{concave}, with $F_i(t_1) = x_1$ and $F_i(t_2)=x_2$. Then some subsequence of the $F_i$ converges uniformly to a concave function $F$ on any interval $[t_1,t_2'], t_2'<t_2$, with $F'(t) \leq \frac{F'(t)}t$, as well as $F(t_1) = x_1$.
\end{lem}
    
\begin{proof}
    Note that $F_i(t) \in [\frac{t_2-t}{t_2-t_1}x_1 + \frac{t-t_1}{t_2-t_1}x_2 ,t\frac{x_1}{t_1}]$, so by enumerating the rationals and repeatedly taking subsequence, we may assume that $F_i(q)$ converge to some $F(q)$ for every rational $q$.

    When restricted to $\mathbb Q$, $F$ must be concave as it is a limit of concave functions, and therefore can be extended continuously in its interior. It is also continuous at $t_1$ since $F_i'(t_1)$ is bounded above.
    
    Since $G_i(t) := t\frac{x_1}{t_1} - F_i(t)$ are nondecreasing functions, $G_i(r)$ must converge to $G(r)$ for any $r \in [t_1,t_2)$, where $G(r) = r\frac{x_1}{t_1} - F(r)$. On an interval $[t_1,t_2']$, for $t_2' < t_2$, we can consider convergence at the partition points 
    
    \begin{equation}
    G^{-1} (G(t_1) + 2^{-n}k(G(t_2')-G(t_1)), k \in [0,2^n] \cap \mathbb Z,
    \end{equation}
    to show that $G_i$ is eventually within $2^{-n}+\varepsilon$ of $G$ on $[t_1,t_2']$. Thus $G_i \rightarrow G$ and hence $F_i \rightarrow F$ uniformly on intervals bounded away from $t_2$. 

    The conditions on $F$ are all inherited from those on $F_i$.
\end{proof}

We introduce some more intermediate lemmas to prove the existence of a minimizing pair $(F,\rho)$.

\begin{lem}\label{rholemma}
    Suppose $(F,\rho)$ is feasible. Then there is a nonincreasing $\hat \rho$ with $(F,\hat \rho)$ feasible, and $I_{2}(F,\hat \rho) = I_{2}(F,\rho)$. Moreover, $\hat \rho$ has the same distribution as $\rho$, in the sense that $\lambda(\rho^{-1}([A,\infty))) = \lambda(\hat \rho^{-1}([A,\infty))), \forall A \geq 0$, and satisfies $\int_{t_1}^t \hat \rho(\tau) d\tau \geq \int_{t_1}^t \rho(\tau) d\tau$ for any $t \in [t_1,t_2]$.
    Here $\lambda$ denotes the Lebesgue measure.
\end{lem}

\begin{proof}
    We define $\hat \rho$ to be the nonincreasing sorted version of $\rho$. In other words, for $\rho: [t_1,t_2] \rightarrow \mathbb R^+$, define

    \begin{equation}
        \nu(A) = \lambda(\rho^{-1}([A,\infty)))
    \end{equation}

    \begin{equation}
        \hat \rho(t) = \inf_{\nu(A) \geq t-t_1} A
    \end{equation}

    It is easy to see that $\hat \rho$ is nonincreasing and has the same distribution as $\rho$.

    Then for any $t \in [t_1,t_2]$, we have

    \begin{equation} \label{intrholem}
        \int_{t_1}^t \rho(\tau) d\tau = \int_0^{\infty} \lambda(\rho^{-1}([A,\infty)) \cap [t_1,t]) dA.
    \end{equation}

    Since $\lambda(\rho^{-1}([A,\infty)) \cap [t_1,t]) \leq \min (\lambda(\rho^{-1}([A,\infty)), \lambda([t_1,t]))$, and since the latter term is just $\lambda(\hat \rho^{-1}([A,\infty)) \cap [t_1,t])$, the integral in \eqref{intrholem} is bounded above by 

    \begin{equation} 
        \int_0^{\infty} \lambda(\hat \rho^{-1}([A,\infty)) \cap [t_1,t]) dA = \int_{t_1}^t \hat \rho(\tau) d\tau.
    \end{equation}
     
    Then 
    \begin{equation}
        g(F,\hat \rho,b)(t) - g(F,\rho,b)(t) = \int_{t_1}^t (\hat \rho - \rho) \geq 0,
    \end{equation}
    so $g(F,\hat \rho,b) \geq g(F, \rho,b) \geq 0$, and
    $I_{2}(F,\rho) = \frac 43 \int \rho^{\frac 32} =  I_{2}(F,\hat \rho)$. 
    
\end{proof}

\begin{lem} \label{dirichlet semicontinuity}
    Suppose $F_i:[A,B] \rightarrow \mathbb R$ are absolutely continuous functions with $\int_A^B F_i^2 < \infty$, converging uniformly to $F$. Then \begin{equation}
        \int_A^B F'^2 \leq \liminf_n \int_A^B F_n'^2.
    \end{equation}
\end{lem}

\begin{proof}
    The Dirichlet norm of $F$, $\int_A^B F^2$, is also given by the supremum over partitions $P$ of $[A,B]$ with breaks at $a=t_0,t_1, \ldots, t_n=B$ of $d_P(F)$, where
    \begin{equation}
        d_P(F) = \sum_{i=1}^n \frac{(F(t_i)-F(t_{i-1}))^2}{t_i-t_{i-1}}.
    \end{equation}
    For any given $P$, we can choose $N$ such that $||F-F_n||_u$ is sufficiently small for $n>N$, and then $d_P(F_n)$ can be made arbitrarily close to $d_P(F)$. Since $\liminf_n \int_A^B F_n^2 \geq \liminf_nd_P(F_n)$, we have $d_P(F) \leq \liminf_n \int_A^B F_n^2$, so the result follows by taking the supremum over all partitions $P$.
\end{proof}

\begin{lem} \label{rhoconverge}
    If $\{ \rho_i \}_{i=1}^{\infty}$ is a sequence of nonincreasing, nonnegative functions on $[A,B]$ with $\int_A^B \rho_i^{\frac 32} < M$, then it admits a subsequence that converges in $L^1$ to some $\rho$, with $\int_A^B \rho^{\frac 32} \leq \liminf \int_A^B \rho_i^{\frac 32}$.
\end{lem}

\begin{proof}
    Note that $\int_A^B \rho_i^{\frac 32} < M$ implies that $\int_A^B \rho_i < L := (B-A)^{\frac 13}M^{\frac 23}$ by Holder's inequality.
    For $\varepsilon>0$, we must have $\rho_i(A+\varepsilon) < \frac L {\varepsilon}$. Then on the interval $[A+\varepsilon, B]$, we have $\rho_i(t) = \rho_i(A+\varepsilon) - \mu_i([A+\varepsilon,t])$ for some positive, uniformly bounded measures $\mu_i$. By passing to a subsequence, we may assume that $\rho_i(A+\varepsilon)$ converge and that $\mu_i$ converge weakly. Then
    \begin{equation}
    \begin{aligned}
        \int_{A+\varepsilon}^B |\rho(t) - \rho_i(t)| &\leq (B-A)(\rho(A+\varepsilon)-\rho_i(A+\varepsilon)) \\ &+ \int_{A+\varepsilon}^B |\mu([A+\varepsilon,t]) - \mu_i([A+\varepsilon,t])| dt
    \end{aligned}
    \end{equation}

    The first term converges to zero, and the integral term is, by Fubini's theorem, equal to
    
    \begin{equation}
        \int_{A+\varepsilon}^B \int_{A+\varepsilon}^t 1 d(\mu-\mu_i)(\tau) dt = \int_{A+\varepsilon}^B (B-\tau) d(\mu -\mu_i)(\tau).
    \end{equation}
    
    This goes to zero since $\mu_i \rightarrow \mu$, so $\rho_i$ converge to $\rho$ in $L^1$ on $(A+\varepsilon, B]$.
    
    By Holder's inequality, $\int_A^{A+\varepsilon} \rho(\tau) d\tau \leq \varepsilon^{\frac 13}M^{\frac 23}$, so $\rho_i \rightarrow \rho$ in $L^1$ on $[A,B]$. Since the unit ball of $L^q([A,B])$ is closed in $L^p([A,B])$ for $p<q$, we have $\int_A^B \rho^{\frac 32}(\tau) d\tau \leq \liminf \int_A^B \rho_i^{\frac 32}(\tau) d\tau$.  
\end{proof}

\begin{lem}\label{optimizer}
    There exists a feasible pair $(F,\rho)$ with $I_{2}(F,\rho) = \inf_{\hat F, \hat \rho} I_{2}(\hat F, \hat \rho)$, where the infimum is taken over all feasible $(\hat F,\hat \rho)$.
\end{lem}

\begin{proof}

    Observe that the set of feasible $(\hat F, \hat \rho)$ is nonempty: if $\hat F$ is the straight line connecting $(x_1,t_1)$ to $(x_2,t_2)$, and $\hat \rho = \frac{x_1^2}{t_1^2} + \frac{x_2^2}{t_2^2}$ on $[t_1,t_2]$, then $(\hat F,\hat \rho)$ is feasible. By Lemma \ref{concave}, there exist feasible pairs $(F_i,\rho_i)_{i=1}^{\infty}$ with $F_i$ concave and $F_i'(t) \leq \frac{F_i(t)}t \forall t \in [t_1,t_2]$, such that $\lim_{i \rightarrow \infty} I_2(F_i,\rho_i) = \inf_{\hat F, \hat \rho} I_{2}(\hat F, \hat \rho)$. By Lemma \ref{rholemma}, we may also assume that each $\rho_i$ is nonincreasing. We will find a converging subsequence whose rate will therefore be minimal. 
    
    Let $M = \max_i I_2(F_i,\rho_i)$, so that $\int_{t_1}^{t_2} \rho_i(\tau)^{\frac 32} d\tau \leq \frac 34 M$. Then $\int_{t_1}^{t_2} \rho(\tau) d\tau \leq (t_2-t_1)^{\frac 13} (\frac 34 M)^{\frac 23}$ by Holder's inequality. Then for $D := (t_2-t_1)^{\frac 13} (\frac 34 M)^{\frac 23} + b$, we must have $\int_{t_1}^{t_2} F_i'(\tau)^2 d\tau \leq D$.

    By Lemma \ref{uniformF}, by passing to a subsequence we may assume that $F_i \rightarrow F$ uniformly on sets bounded away from $t_2$. Moreover, for $\varepsilon>0$, if $\delta = \frac{\varepsilon^2}D$, we must have $|F_i(t) - x_2| \leq \varepsilon$ whenever $|t-t_2| \leq \delta$. Thus $F_i$ converge uniformly to $F$.

    Similarly, by Lemma \ref{rhoconverge}, by passing to a subsequence again we may assume that $\rho_i \rightarrow \rho$ in $L^1$ and that $\int \rho^{\frac 32} = \frac 34 \inf_{\hat F, \hat \rho} I_{2}(\hat F, \hat \rho)$.

    It follows by Lemma \ref{dirichlet semicontinuity} that $g(F,\rho,b)(t) \geq \limsup g(F_i,\rho_i,b)(t) \geq 0$, so $(F,\rho)$ is feasible, and moreover that 
    
    \begin{equation}
        I_{2}(F,\rho) = \frac 43 \int_{t_1}^{t_2}\rho^{\frac 32}(\tau) d \tau = \lim_{i \rightarrow \infty} \frac 43\int_{t_1}^{t_2}\rho_i^{\frac 32}(\tau) d \tau = \inf_{\hat F, \hat \rho} I_{2}(\hat F, \hat \rho).
    \end{equation}

\end{proof}

\begin{lem}\label{tightg}
    Suppose $(F,\rho)$ is feasible and minimizes $I_{2}$. Then $g(F,\rho,b)(t_2)=0$. In particular, $F$ is concave and $ F'(t) \leq \frac{ F(t)}t$ for all $t \in [t_1,t_2]$.
\end{lem}

\begin{proof}
    Since $b \leq \frac{(x_2-x_1)^2}{t_2-t_1} + \frac{x_1^2}{t_1} - \frac{x_2^2}{t_2}$, if $\rho$ is zero almost everywhere, then $b = \frac{(x_2-x_1)^2}{t_2-t_1} + \frac{x_1^2}{t_1} - \frac{x_2^2}{t_2}$ and $F$ is linear on $[t_1,t_2]$ and we compute $g(t_2)=0$. Otherwise, $\int \rho > 0$, and suppose for the sake of contradiction that $g(t_2)>0$. The integral $\int_t^{t_2} \rho(\tau)d\tau$ is a continuous function of $t$ decreasing to zero, so fix $t$ sufficiently close to $t_2$ so that $0 < \int_t^{t_2} \rho(\tau)d\tau < g(t_2)$. From \eqref{gprime}, since $g'(\tau) \leq \rho(\tau)$, we will show that $(F,\hat \rho := \chi_{_{[t_1,t]}} \rho)$ is feasible: if $\hat g = g(F,\hat \rho,b)$, then we have $\hat g(t') = g(t')$ for $t' \leq t$, and for $t' \in [t,t_2]$,
    
    \begin{equation}
        \hat g(t') \geq \hat g(t_2) = g(t_2) - \int_t^{t_2}\rho(\tau)d\tau > 0,
    \end{equation}
    
    so $(F,\hat \rho)$ is feasible. But then $I_{2}(F,\hat \rho) < I_2(F,\rho)$, contradicting minimality of $I_{2}(F,\rho)$.

    If $F$ is not of the stated form, then by Lemma \ref{concave} there exists $\hat F$ with $(\hat F, \rho)$ also minimizing $I_{2}$ such that $g(\hat F, \rho, b)(t_2) > g(F, \rho, b)(t_2) \geq 0$, which is a contradiction.

\end{proof}

From this point forth, we will assume that $(F,\rho)$ is a feasible pair minimizing $I_{2}$ and satisfying the conclusion of Lemma \ref{concave}, and that $g$ is the function arising from the definition in $\eqref{defineg}$. We will also assume by Lemma \ref{rholemma} that $\rho$ is nonincreasing.

\begin{lem} \label{line}
    Suppose for some interval $(p,q)  \subset (t_1,t_2) $, $g(t)>0, \forall t \in (p,q)$. Then $F$ is linear and $\rho$ is constant on $(p,q)$.
\end{lem}
\begin{proof}
    Fix $t \in (p,q)$.
    Pick $\varepsilon < \min( \frac{t_2-t}2, \frac{g(t)}4 (\frac{\max F}{t_1} - \min_{t' < \frac{t+t_2}2} F'(t'))^{-2})$, sufficiently small so that $\int_{t-\varepsilon}^t \rho(\tau) d\tau < \frac{g(t)}2$. Suppose $F$ is not linear on $[t-\varepsilon, t+\varepsilon]$. Let $G=F$ outside of $(t-\varepsilon,t+\varepsilon)$, and define $G$ linearly on $[t-\varepsilon,t+\varepsilon]$. For $t \in (t-\varepsilon, t+\varepsilon)$, we have

    \begin{equation} \label{thistemponehere}
        g(G,\rho,b)(t') \geq g(F,\rho,b)(t') - \int_{t-\varepsilon}^{t'} (\frac{G(\tau)}{\tau} - G'(\tau))^2 d\tau
    \end{equation}

    Observe that $\frac{G(\tau)}{\tau} \geq G'(\tau)$ on $(t-\varepsilon,t+\varepsilon)$: $G'(\tau)$ is constant, so we just need to show that $G$ does not cross the parallel line $x=\tau G'$ on this interval. But by Lemma \ref{concave}, 

    \begin{equation}
    G(t-\varepsilon) =F(t-\varepsilon) \geq (t-\varepsilon)F'(t-\varepsilon)\geq(t-\varepsilon)G'.    
    \end{equation}
    
    Next, since $(-\frac{F(t)}t - F'(t))^2 \leq g'(t) \leq \rho(t)$ by \eqref{gprime}, equation \eqref{thistemponehere} is bounded below by

    \begin{equation}
    g(F,\rho,b)(t) - \int_{t-\varepsilon}^t\rho(\tau)d\tau - \int_{t-\varepsilon}^{t+\varepsilon}(\frac{G(\tau)}{\tau} - G'(\tau))^2 d\tau >g(t)-\frac{g(t)}2-\frac{g(t)}2=0,
    \end{equation}
    
    where the inequality on the second integral comes from taking an upper and lower bound on $G$ and $G'$, as in the choice of $\varepsilon$. By choice of $\varepsilon$, we have $g(G,\rho,b)(t')>0$ for $t' \in (t-\varepsilon,t+\varepsilon)$, and since the straight line has the shortest Dirichlet length, $g(G,\rho,b)(t')>g(F,\rho,b)(t')$ for $t' > t+\varepsilon$. But then $(G,\rho)$ is feasible and minimizes $I_2$, while $g(G,\rho,b)(t_2)>0$, in contradiction to Lemma \ref{tightg}.
    
    Now suppose $\rho$ is not constant on $[t-\varepsilon, t+\varepsilon]$, and let $\lambda$ be the average value of $\rho$ on the interval. If, for $\delta \in (0,1)$, we define $\rho_{\delta}$ equal to $(1-\delta)\rho + \delta \lambda$ on $[t-\varepsilon, t+\varepsilon]$ and equal to $\rho$ otherwise, then $I_2(F,\rho_{\delta}) < I_2(F,\rho)$ due to convexity of $x \mapsto x^{\frac 32}$. But $g(F,\rho,b) = g(F,\rho_{\delta},b)$ outside of the interval, and for sufficiently small $\delta$, $g(F,\rho_{\delta},b) > 0$ is positive on the interval $[t-\varepsilon, t+\varepsilon]$ as well, contradicting minimality of $(F,\rho)$.
    
    Thus, there is some neighborhood of $t$ on which $F$ is linear and $\rho$ is constant.
    
    Considering various $t \in (p,q)$, these neighborhoods make an open cover of $[p+\varepsilon, q-\varepsilon]$, which has a finite subcover. Since line segments on overlapping open sets must have the same slope, $F$ is linear on $(p,q)$. Similarly, $\rho$ is constant on $(p,q)$.
\end{proof}

\begin{lem}\label{curvecont}
    If $g=0$ on the interval $[s_1,s_2]$, then $F'$ and $\rho$ are absolutely continuous on these intervals.
\end{lem}

\begin{proof} 
    For $s \in [s_1,s_2]$, we have $g'(s)=0$, so, almost surely
    
    \begin{equation} \label{equationforthislemma}
        \rho(s) - (\frac{F(s)}s-F'(s))^2 =0.
    \end{equation}
    $\rho$ is nonincreasing, so it is either continuous or has downward jump discontinuities. Similarly, $F'$ is nonincreasing, so $\frac{F(s)}s-F'(s)$ has only upward jump discontinuities. Since $F'(t) \leq \frac{F(t)}t$ by Lemma \ref{concave}, $(\frac{F(s)}s-F'(s))^2$ has only upward jump discontinuities. Then the left side of \eqref{equationforthislemma} can have only downward jump discontinuities. But the right side of \eqref{equationforthislemma} is equal to the constant zero, so $\rho$ and $F'$ cannot have any discontinuities.
    
    Next, \eqref{equationforthislemma} can be rewritten as

    \begin{equation} \label{otherequationforthislemma}
        \frac{F(s)}s = \sqrt{\rho(s)} + F'(s).
    \end{equation}

    Since $\rho$ and $F'$ are continuous and nondecreasing, we can define nonnegative measures $\mu_1$ and $\mu_2$ by

    \begin{equation}
    \sqrt{\rho(s)} = \sqrt{\rho(s_1)} + \mu_1([s_1,s]), F'(s) = F'(s_1) + \mu_2([s_1,s]). 
    \end{equation}

    But since $\frac{F(s)}s$ is absolutely continuous, by \eqref{otherequationforthislemma}, $\mu_1+\mu_2$ is continuous with respect to the lebesgue measure and hence $\mu_1,\mu_2 << \lambda$. Thus $F'$ and $\sqrt{\rho}$ are absolutely continuous, and therefore so is $\rho$.
    
\end{proof}

\begin{lem} \label{curve}
    If $g(t) = 0$ for $t \in [s_1,s_2]$, then on this interval, $F(t) = c_1 \sqrt t + c_2 t$, and $\rho(t) = (\frac{F(t)}t - F'(t))^2$.
\end{lem}

The Euler-Lagrange equation from calculus of variations states that the function $f = \rho^{\frac 32}$ must satisfy the equation

\begin{equation}
    \frac{\partial f}{\partial F} = \frac{d}{dt}\frac{\partial f}{\partial F'}.
\end{equation}

We give a quick alternative way to see this through calculus.

\begin{proof}
    The equality $\rho(t) = (\frac{F(t)}t - F'(t))^2$ follows directly from the definition of $g$. Let $\phi$ be a nonnegative smooth function supported on $[s_1,s_2]$, and define 
    \begin{equation}
    F_{\alpha} = F + \alpha \phi,  \rho_{\alpha}(t) = (\frac{F_{\alpha}(t)}t - F_{\alpha}'(t))^2.
    \end{equation}
     Now consider the effect of $\alpha$ on $I_{2} = \int \rho^{\frac 32} = \int (\frac{F(s)}{s}-F')^3$:
    \begin{equation}
    \begin{aligned}
    \frac d{d\alpha} (\frac 34 I_{2}(F_{\alpha},\rho_{\alpha}))\Bigr|_{\alpha = 0} 
    &= \frac d{d\alpha}\int_{s_1}^{s_2}(\frac{F_{\alpha}(s)}s - F_{\alpha}'(s))^3 ds \\
    &= \int_{s_1}^{s_2} \frac{\partial F_{\alpha}(s)}{\partial \alpha}\frac{\partial}{\partial F(s)}(\frac{F(s)}s - F'(s))^3 ds \\ &+ \int_{s_1}^{s_2} \frac{\partial F_{\alpha}'(s)}{\partial \alpha}\frac{\partial}{\partial F'(s)}(\frac{F(s)}s - F'(s))^3 ds 
    \end{aligned}
    \end{equation}
    Using the linear dependence of $F_{\alpha}$ on $\alpha$, we see this is equal to
    \begin{equation}
    \begin{aligned}
    &\int_{s_1}^{s_2} (\phi(s-t))3(\frac{F(s)}s - F'(s))^2(\frac 1s) ds \\ + &\int_{s_1}^{s_2} (\phi'(s-t))3(\frac{F(s)}s - F'(s))^2(-1) ds 
    \end{aligned}
    \end{equation}
    This derivative must be equal to zero for $F = F_0$ and for any $\phi$, as otherwise $I_{2}$ will be lower for $F_{\alpha}$ for some $\alpha$ sufficiently close to zero. Since $\phi$ was arbitrary, this means that $3(\frac{F(s)}s - F'(s))^2(\frac 1s)$ is the weak derivative of $3(\frac{F(s)}s - F'(s))^2(-1)$. Since the latter is absolutely continuous by Lemma \ref{curvecont}, we can infer the differential equation
    \begin{equation}
    \begin{aligned}
    -\frac{d}{ds}(\frac{F(s)}s - F'(s))^2 = \frac 1s (\frac{F(s)}s - F'(s))^2.
    \end{aligned}
    \end{equation}

    This simplifies to 

    \begin{equation} \label{diffeq}-(\frac{F'(s)}s - \frac{F(s)}{s^2} - 2F''(s))(\frac{F(s)}s - F'(s)) = 0.
    \end{equation}

    Now suppose $F'(t) = \frac{F(t)}t$ for some $t \in (s_1,s_2)$. Since $\rho$ is nonincreasing by Lemma \ref{rholemma}, this implies that $\rho = 0$ on $[t,s_2]$, so $F'(s_2) = \frac{F(s_2)}{s_2}$. By Lemma \ref{concave}, $F$ is concave, so $F(t_1) \leq F(s_2) - \frac{F(s_2)}{s_2}s_2(s_2-t_1)$, and therefore $F(s_2) \geq x_1\frac{s_2}{t_1}$. But by Lemma \ref{derivativelembad}, $\frac{F(s_2)}{s_2}$ is nonincreasing, so $F'(s) = \frac{F(s)}s = \frac{x_1}{t_1}$ must hold for any $s \in [t_1,s_2]$. Thus $F$ is linear on $[s_1,s_2]$, which satisfies the conclusion of this lemma with $c_1=0$.
    
    Otherwise, we can divide \eqref{diffeq} by $-(\frac{F(s)}s - F'(s))$, and we get the following differential equation:

    \begin{equation}
    -\frac {F(t)}{t^2}+\frac{F'(t)}t - 2F''(t) = 0.
    \end{equation}
    Since we know the antiderivative, $(\frac{F(t)}{t}-2F'(t))$, is absolutely continuous, it must be equal to some constant $C_1$. We can then divide by $\sqrt t$ to get
    \begin{equation}
    \frac{F(t)}{t^{\frac 32}}-2\frac{F'(t)}{\sqrt t} = \frac{C_1}{\sqrt t}.
    \end{equation}

    Again the antiderivative of the left side, $-\frac{2F(t)}{\sqrt t}$, is absolutely continuous, and so integrating again yields
    
    \begin{equation}
    F(t) = -C_1 \sqrt t -2C_2t.
    \end{equation}
\end{proof}

\begin{lem} \label{zero}
    There exists $t_{B} \in [t_1,t_2]$ such that $g(t)=0$ for $t \in [t_{B},t_2]$, and $g(t) > 0$ for $t \in [t_1,t_{B})$. $t_{B} = t_1$ if and only if $b=0$.
\end{lem}

\begin{proof}
    We have $g(t_2)=0$ by Lemma \ref{tightg}. Suppose the claim is false; then there exists some $s \in [t_1,t_2]$ such that $g(s) > 0$, but $g$ has a zero both above and below $s$. Let $p$ be the greatest zero below $s$ and $q$ be the least zero above $s$, so that $g(t)>0$ for $t \in (p,q)$, and $g(p)=g(q)=0$. By Lemma \ref{line}, $\rho$ is constant and $F$ is linear on $(p,q)$, so $F'(t)$ is also constant. By Lemma \ref{derivativelembad}, $\frac{F(t)}t$ is nonincreasing, so this means that $g'(t) = \rho(t) - (\frac{F(t)}t-F'(t))^2$ is nondecreasing, which is impossible as $g(p)=0$, $g(s)>0$, $g(q)=0$. $t_{B}$ is the smallest $t$ such that $g(t)=0$, so $t_{B}=t_1$ if and only if $g(t_1)=b=0$.
\end{proof}

\begin{proof}[Proof of Proposition \ref{bigtheorem}]By Lemma \ref{optimizer}, there exists a minimizing feasible pair for $I_{2}$. Suppose $(F_1,\rho_1)$ and $(F_2,\rho_2)$ are both minimizers. Since $g$ is concave as a function of $F$ and linear in $\rho$, it must be that $(\frac 12(F_1+F_2),\frac 12(\rho_1+\rho_2))$ is feasible. Since $\int \rho^{\frac 32}$ is strictly convex, it follows that $I_{2}(\frac 12(F_1+F_2),\frac 12(\rho_1+\rho_2)) \leq I_{2}(F_1,\rho_1)$, with equality possible only when $\rho_1=\rho_2$ almost everywhere. Therefore, the optimal $\rho$ is unique.

     By Lemma \ref{tightg}, any optimizing $F$ must satisfy the conclusions of Lemma \ref{concave}. By Lemmas \ref{line}, \ref{curve}, \ref{zero}, $F$ must be of the given form: namely, there exists $t_{B} \in [t_1,t_2]$ such that $F$ is linear and $\rho$ is constant on $[t_1,t_{B})$, and $F(t) = c_1 \sqrt t + c_2 t$ and $\rho(t) = (\frac{F(t)}s-F'(t))^2$ for $t \in (t_{B},t_2]$, with $t_{B}=t_1$ if and only if $b=0$, and moreover, $g', F', \rho$ are continuous at the merge point $t_{B}$.

     The fact that the parameters $c_1,c_2,x_{B},t_{B}$ are unique, and therefore that $F$ is unique, follows from the uniqueness of $\rho$.

\end{proof}

This completes the answer to Problem \ref{relaxedstatement}.

\section{Main problem and computations}

We now apply Problem \ref{relaxedstatement} to the original question of finding the lowest rate metric for which the geodesic from $(0,0)$ to $(0,1)$ passes through $(1,a)$. This will be an extension of $F$ and $\rho$ from $[a,1]$ to $[0,1]$.

\begin{lem} \label{caseslemma}
    Suppose $(F,\rho)$ answer Problem \ref{relaxedstatement} for $(x_2,t_2)=(0,1)$ and for $(x_1,t_1)=(1,a)$. Extend $F$ and $\rho$ to $[0,1]$ by defining $F(0)=0$, defining $F$ linearly on $[0,a]$ and setting $\rho$ to be constant on $[0,a]$ with the value such that $b = \int_0^a (\rho(t)-(\frac{F(t)}{t}-F'(t))^2)dt$. Suppose that $b$ is such that $\rho$ is constant on $[0,t_{B}]$. Then 
    \begin{equation} \label{casescondition}
        -\frac{(F(q)-F(p))^2}{q-p} \leq \int_p^q \rho(t)-F'(t)^2 dt,\forall p<q \in [0,1],
    \end{equation}
    with strict inequality if $a \not \in (p,q)$ and $F$ is not linear on $[p,q]$.
\end{lem}

\begin{proof}

    We remark that this condition is equivalent to the path along $F$ having greater length than the straight line shortcut, with strict inequality whenever the shortcut is not equal to $F$ and does not pass through $a$, and proceed by cases:

    \begin{itemize}
        \item Case 1. $p,q \in [0,a]$ or $p,q \in [a,t_{B}]$. 
        
        This case is trivial because $F$ is linear on $[p,q]$, so the difference between the left and right sides of \eqref{casescondition} is $\int_p^q \rho(t)dt \geq 0$.
        \item Case 2. $p \in [0,a)$, $q \in (a,t_{B}]$.

    Suppose \eqref{casescondition} does not hold. If we define $F_L$ equal to $F$ on $[0,t_{B}]$ and extend the line from $(1,a)$ to $(x_{B},t_{B})$ for $t>t_{B}$, and define $\rho_L$ to be the constant value that $\rho$ takes on $[0,t_{B}]$, then by rescaling linearly we find that there is also a straight line shortcut for $(F_L,\rho_L)$ from $(0,0)$ to $(F_L(r),r)$ for $r=a+\frac{a}{a-p}(q-a)$.
    Observe that $g(F_L,\rho_L,b) = g(F,\rho,b)$ on $[0,t_{B}]$, that $g(t_{B})=0$ by definition and that $g'(t_{B})=0$ because $g'$ is continuous due to Proposition \ref{bigtheorem} and $t_{B}$ is a global minimum for $g(F,\rho,b)$. Considering \eqref{casescondition}, we have supposed the difference between the right and left sides, which is equal to $g(F_L,\rho_L,b)(r)$, is negative. Notice that 

    \begin{equation}
        g(F_L,\rho_L,b)(r) = \int_0^r \rho(t)-F'(t)^2 dt + \frac{(F(r)^2)}{r} 
    \end{equation}

    is quadratic in $r$, so $g(t_{B})=g'(t_{B})=0$ implies that it is everywhere nonnegative. This contradicts the existence of a shortcut for $(F_L,\rho_L)$ from $(0,0)$ to $(F(r),r)$.
    \item Case 3. $q \in (t_{B},1)$.
    
    Define $g_p(F,\rho)$ to be the difference in the sides of \eqref{casescondition}, so that
    \begin{equation}
    g_p(t) = \int_p^t \rho(\tau)-F'(\tau)^2 d\tau + \frac{(F(t)-F(p))^2}{t-p}.
    \end{equation}

    It follows that $g_p'(t) = \rho(t) - (F'(t)-\frac{F(t)-F(p)}{t-p})^2$

    We have seen in case 2 that when $p<t_{B}$, we have $g_p(t_{B}) \geq 0$. Therefore, for $r = \max(p,t_{B})$, we have $g_p(r) \geq 0$.

    Consider $t>r$. Notice that $\frac{F(t)-F(p)}{t-p} = \frac{1}{t-p}\int_p^t F'(\tau)d\tau$ is the average value of the nonincreasing function $F'$. It follows that

    \begin{equation}
    F'(t) \leq \frac{F(t)-F(p)}{t-p} \leq \frac{F(t)-F(0)}{t-0}
    \end{equation}

    where the second inequality is strict when $p>0$ since $F'$ is not constant on $[0,t]$.

    Consequently, $g'_p(t) \geq g'_0(t) = g'(t) = 0$ for $t \in [r,1]$. Therefore, $g_p(q) \geq 0$, so \eqref{casescondition} is satisfied.
    \end{itemize}  

    Note that equality holds only in Case 2, when $r = t_{B}$ and hence $p<a<q$, and in Case 3, when $p=0$ and hence $p<a<q$.
\end{proof}

\begin{thm} \label{geodesicshape}
    Consider the set of directed metrics for which the geodesic from $(0,0)$ to $(0,1)$ passes through $(1,a)$, for $a<\frac 12$. This set has a unique metric of minimal rate, $e(F,\rho)$, where $F,\rho$ are described as follows:
    
    The geodesic $F$ consists of a line segment from $(0,0)$ to $(1,a)$, followed by a line segment from $(1,a)$ to some point $(x_{B},t_{B})$, where $a < t_{B}$ and $x_{B} < \frac{t_{B}}a$, followed by the parabola $F(t) = x_{B}\frac{\sqrt t - t}{\sqrt t_{B} - t_{B}}$ for $t \in [t_{B},1]$. The planted density
    $\rho$ is constant on $[0,t_{B}]$ and $\rho(t) = (\frac{F(t)}t-F'(t))^2$ for $t \in [t_{B},t_2]$. $F'$ and $\rho$ are continuous at $t_{B}$. $0 < b \leq \frac{1}{1-a} + \frac{1}{a} = \frac{(x_2-x_1)^2}{t_2-t_1} + \frac{x_1^2}{t_1} - \frac{x_2^2}{t_2}$.
\end{thm}

\begin{proof}
    We will first show that this pair of $F,\rho : [0,1] \rightarrow \mathbb R$ is the best among those that admit no shortcuts from $(0,0)$, and then show that $e(F,\rho)$ also admits no other shortcuts.
    
    For a given value of $b$, we consider the values of $F$ and $\rho$ on $[0,a]$. We have
    
    \begin{equation} \label{intrhoisb}
        b = \int_0^a \rho(t) - F'(t)^2 dt + \frac{1}{a} = \int_0^a \rho(t) - (F'(t) - \frac 1a)^2 dt,
    \end{equation}

    where we used the fact that $\int_0^a F'(t)dt = 1$.
    
    It follows that the contribution to the rate from $t \in [0,a]$, $I_{1}$, is minimized when $F$ is a straight line and $\rho$ is constant on $[0,a]$: from \eqref{intrhoisb} we see that $F' = \frac 1a$ uniquely minimizes the value of $\int_0^a \rho(t)dt$, and since $x \mapsto x^{\frac 32}$ is convex, $I_{1} := \int_0^a \frac 43 \rho^{\frac 32}$ is minimized when $\rho$ is constant on this line by Jensen's inequality. The minimum value of $I_{1}$ is 

    \begin{equation}
        I_{1} = \frac 43 a \rho^{\frac 32} = \frac 43 b^{\frac 32} a^{-\frac 12}.
    \end{equation}
    
    Clearly, $I_{1}$ is convex as a function of $b$. Observe that if $(F_1,\rho_1)$ and $(F_2,\rho_2)$ are feasible for Problem \ref{relaxedstatement} for $b_1$ and $b_2$, respectively, then for $\lambda \in (0,1)$, $(\lambda F_1 + (1-\lambda)F_2,\lambda \rho_1 + (1-\lambda)\rho_2)$ is feasible for $\lambda b_1 +(1-\lambda) b_2$. Applying this to minimizing $(F_1,\rho_1)$ and $(F_2,\rho_2)$, since $I_2$ is a convex function of $\rho$, we see that the minimum value of $I_2$ is a convex function of $b$. Then the minimum value of $I := I_1+I_2$ is convex as a function of $b$ which goes to infinity as $b \rightarrow \infty$, so $I$ is minimized at some unique finite $b=b_{opt}$.
    
    This $b_{opt}$ must satisfy $b_{opt} \leq \frac{(x_2-x_1)^2}{t_2-t_1} + \frac{x_1^2}{t_1} - \frac{x_2^2}{t_2}$ as in the condition of Problem \ref{relaxedstatement}, since $I_{2}$ is uniformly zero for $b \geq \frac{(x_2-x_1)^2}{t_2-t_1} + \frac{x_1^2}{t_1} - \frac{x_2^2}{t_2}$, while $I_{1}$ is strictly increasing in $b$.

    Consider now the pair $(F,\rho)$ minimizing $I$ for $b=b_{opt}$. We first want to rule out the degenerate case where $b_{opt}=0$. Suppose otherwise; then $\int_0^a \rho(t)dt=0$, so $\rho$ is uniformly zero on $[0,a] = [0,t_{B}]$. Define $\hat \rho(t) = (1-a)\rho(a+\frac{t}{1-a})$. Then
    \begin{equation}
        \int_0^t \hat \rho(\tau) d\tau = \int_0^{a+\frac{t}{1-a}} \hat \rho(\tau) d\tau \geq \int_0^a \rho(\tau) d\tau,
    \end{equation}
    so $(F,\hat \rho)$ is feasible for Problem \ref{relaxedstatement} for some $\hat b>0$ (note that $\hat b = \int_0^a \hat \rho \neq 0$ since $\rho$ is nonincreasing on $[a,1]$). Moreover, $I(F,\hat \rho) = (1-a)^{\frac 12}I(F,\rho)$, so the minimizer for $\hat b$ will have lower rate than the minimizer for $b=0$, and therefore $b_{opt} \neq 0$.

    Thus $g(a)>0$, so in particular by Lemma \ref{zero}, $t_{B}>a$. We know that $\frac{F(t_{B})}{t_{B}} \leq \frac{a}{1}$ by Lemma \ref{derivativelembad}, and equality would imply that $g'(t) = \rho \geq 0$ for $t \in (a,t_{B})$, which contradicts $g(t_{B})=0$. Therefore, $x_{B} < \frac{t_B}a$.
    
     Next, we show that $\rho$ is constant on $[0,t_{B}]$. Note that $\rho$ and $F'$ are constant on $[0,a]$ and on $[a,t_{B}]$. Suppose $\rho$ is not constant on $[0,t_{B}]$, and let $k$ be the average value of $\rho$ on $[0,t_{B}]$. Then for $\varepsilon < \frac{g(a)}{t_{B}k}$, we can write $\hat \rho = (1-\varepsilon)\rho +\varepsilon k$ on $[0,t_{B}]$ and $\hat \rho = \rho$ otherwise, so that $(F,\hat \rho)$ is feasible for some $b$. But $\int \hat \rho ^{\frac 32} < \int \rho ^{\frac 32}$ by convexity of $x \mapsto x^{\frac 32}$, which contradicts choice of $b_{opt}$. Therefore, $\rho$ is constant on $[0,t_{B}]$.
    
    Finally, $F$ is in fact a geodesic for $e(F,\rho)$ by Lemma \ref{caseslemma}.
\end{proof}

It remains to compute $I$ and the point $(x_{B},t_{B})$ as a function of $a$.

\begin{thm} \label{ratecomputation}
$I = \frac{8}{3a^2(3-2\sqrt{2a})^2}$. $x_{B} = \frac{4(1-\sqrt {2a})}{3-2\sqrt{2a}}$ and $t_{B} = 2a$.
In particular, in the limit as $a \rightarrow 0$, $x_{B} \rightarrow \frac 43$ and $I = \frac{8}{27a^2} + o(\frac 1{a^2})$.
\end{thm}

\begin{proof}

We use the continuity conditions in Proposition \ref{bigtheorem}. We know that $\rho(t_{B}) = (\frac{F(t_{B})}{t_{B}} - F'(t_{B}))^2$ and $g(t_{B})=0$. Then
\begin{equation}
\begin{aligned}
        0 = g(t_{B}) &= t_{B}(\frac{F(t_{B})}{t_{B}} - F'(t_{B}))^2 + \frac{x_{B}^2}{t_{B}} - \frac{(1-x_{B})^2}{t_{B}-a} - \frac 1a \\ &= t_{B}(\frac {x_{B}}{t_{B}} + \frac{1-x_{B}}{t_{B}-a})^2 + \frac{x_{B}^2(t_{B}-a)a - (1-x_{B})^2t_{B}a - t_{B}(t_{B}-a)}{a(t_{B}-a)t_{B}} \\ &= \frac{(t_{B}-ax_{B})^2}{t_{B}(t_{B}-a)}(\frac1{t_{B}-a} - \frac 1a).
\end{aligned}
\end{equation}

By Theorem \ref{geodesicshape}, $t_{B}=ax_{B}$ is not a solution. Therefore, $t_{B}=2a$.

We now find $x_{B}$ by using continuity of $F'$. We have $F'(t_{B}^-) = \frac{x_{B}-1}a$, while $F'(t_{B}^+) = \frac d{dt}(x_{B}\frac{\sqrt t - t}{\sqrt{2a} - 2a})|_{s=2a} = x_{B}\frac{\frac 1{2 \sqrt {2a}} - 1}{\sqrt{2a}-2a}$.

Equating these yields

\begin{equation}
x_{B} = \frac{4(1-\sqrt {2a})}{3-2\sqrt{2a}}.
\end{equation}

We now use these to compute $I$.

    \begin{equation}
\begin{aligned}
        \frac 34 I &= \int_0^{t_{B}} \rho(s)^{\frac 32}ds + \int_{t_{B}}^1 \rho(s)^{\frac 32}ds \\
        &= t_{B}(\frac{t_{B}-ax}{t_{B}(t_{B}-a)})^3 + \int_{t_{B}}^1 (\frac 1s \frac{x_{B}(\sqrt s - s)}{\sqrt {t_{B}} - t_{B}} - x_{B}\frac{\frac1 {2\sqrt s}-1}{\sqrt {t_{B}} - t_{B}})^3 ds \\
        &= \frac{2}{a^2(3-2\sqrt{2a})^2}.
\end{aligned}
\end{equation}

\end{proof}

\section{Proofs of main theorems}

The main results of the paper now follow from usual analysis using rate functions.

\begin{proof}[Proof of Theorem \ref{metrictheorem}]
Let $I_{min} = \frac {8}{3a^2(3-\sqrt{8a})^2}$

Denote by $S$ the set of directed metrics for which the rightmost geodesic from $(0,0)$ to $(0,1)$ passes through $(z,a)$ for some $z \geq 1$. We will show that this set is relatively closed in the set of finite rate metrics in the topology of uniform convergence on bounded sets. In other words, if $\{e_n\}_{n=1}^{\infty}$ are finite rate metrics converging to a finite rate metric $e$, and if each $e_n$ has a geodesic passing through $(z_n,a)$ for some $z_n \geq 1$, then $e$ also has a geodesic passing through $(z,a)$ for some $z \geq 1$. 

By Lemma 4.10 in \cite{Das_Dauvergne_Virag_2024}, observe that for sufficiently large $N>0$ and for some $M>0$, we have $\sup_{t \in (0,1)} e(0,0,N,t)+e(N,t,0,1) < -M$. 

If infinitely many $e_n$ have a geodesic from $(0,0)$ to $(0,1)$ passing through some $(z,s)$ for $z > N$, then $\sup_{t \in (0,1)} e_n(0,0,N,t)+e_n(N,t,0,1) \geq 0$, which contradicts $e_n \rightarrow e$. Thus we may assume that the geodesics from $(0,0)$ to $(0,1)$ of each $e_n$ do not exceed $N$. But then each $e_n$ satisfies 

\begin{equation}
    \sup_{ 1 \leq z \leq N} (e_n(0,0,z,a)+e_n(z,a,0,1))=e_n(0,0,0,1).
\end{equation}
Then $e_n(0,0,0,1)-e_n(0,0,z,a)-e_n(z,a,0,1)$ are continuous functions of $z$ whose minimums are zero. Then $e(0,0,0,1)-e(0,0,z,a)-e(z,a,0,1)$, being their uniform limit, must also have a minimum of zero. This proves that $e$ also has a rightmost geodesic passing through $(z,a)$ for some $z \geq 1$.

Now let $e=e(F,\rho)$ be the unique minimizing metric from Theorem \ref{geodesicshape}. Let $U$ be an open set containing $e$. By picking arbitrary finite rate $\tilde e \in U^C \cap S$, we see that $I^{-1}([0,I(\tilde e)])$ is compact because $I$ is a good rate function. 

Since $U^C \cap (S \cap I^{-1}([0,I(\tilde e)]))$ is an intersection of closed sets and is contained in a compact set, it is compact. Since $I$ is lower semi-continuous, $I$ must attain its minimum on some $\hat e \in U^C \cap S \cap I^{-1}([0,I(\tilde e)])$, which will be its minimum on $U^C \cap S$. By Theorem \ref{geodesicshape}, this minimum value is greater than $I_{min}$. Then by definition of the rate function, for some $\delta_1 > 0$, we have

\begin{equation} \label{lessthan}
    P(\mathcal L_{\varepsilon} \in U^C \cap S) \leq \exp(-(I_{min}+\delta_1+o(1))\varepsilon^{-\frac 32}).
\end{equation}

On the other hand, $U$ contains $e(\hat F,\hat \rho)$, where $\hat F(t) = (1+\lambda)F(t)$ and $\hat \rho(t) = (1+\lambda)^2\rho(t)+\lambda$. Then the length of the geodesic passing through $(1+\lambda,a)$ will be bounded away from the length of any path passing through $(a,z)$ with $z<1$, so this metric is in the interior of $S$, and hence the interior of $U \cap S$. We can pick $\lambda$ sufficiently small so that $I(\hat F, \hat \rho) = I_{min} + \delta_2$ with $\delta_2 < \delta_1$, and then

\begin{equation} \label{greaterthan}
    P(\mathcal L_{\varepsilon} \in U \cap S) \geq \exp(-(I_{min}+\delta_2+o(1))\varepsilon^{-\frac 32}).
\end{equation}

Since $\mathcal M_{\varepsilon}$ is $\mathcal L_{\varepsilon}$ conditioned on being inside $S$, picking $\varepsilon$ sufficiently small and comparing \eqref{lessthan} and \eqref{greaterthan} gives the desired result.

\end{proof}

\eqref{greaterthan}, together with Theorem \ref{ratecomputation}, also imply Theorem \ref{ratetheorem}.

\begin{proof}[Proof of Theorem \ref{curveshapetheorem}]

By Theorem \ref{metrictheorem}, it suffices to show that uniform convergence on bounded sets to $e=e(F,\rho)$ implies uniform convergence of rightmost geodesics to $F$. We will prove the following slightly stronger statement: for $\varepsilon > 0$, there exists an open set $U$ containing $e$ such that for any $\hat e \in U$, $\hat e$ admits no geodesics $\hat \gamma$ with $\hat \gamma(a) \geq 1$, unless $\hat \gamma$ satisfies $|\hat \gamma - F|_u < \varepsilon$. We will do this by finding a negative upper bound on the length of $\hat \gamma$ in $e=e(F,\rho)$, by considering very large $\hat \gamma$, then $\hat \gamma$ which are very close to $F$ except possibly at $0,a,1$, and then generic $\hat \gamma$.

By Lemma 4.10 in \cite{Das_Dauvergne_Virag_2024}, fix $N$ large enough so that $e(0,0,x,t)+e(x,t,0,1) < -1$ for $|x| \geq N$. Then if $\hat \gamma$ passes through any such $(x,t)$, its length is at most $-1$. This is the first case.

For the second case, fix $\delta > 0$ so that $\delta \leq \frac{(\varepsilon/2 - \delta \sup|F'|)^2}{1+\int_0^1 \rho}$ and suppose $|\hat \gamma(t) - F(t)| < \frac{\varepsilon}{2}$ for $\min(|t|,|t-a|,|t-1|) \geq \frac{\delta}2$, but that $|\hat \gamma - F|_u \geq \varepsilon$. Then $\hat \gamma$ changes by at least $\frac{\varepsilon}2 - \delta \sup|F'|$ over an interval of length at most $\delta$, so its Dirichlet length is at most $-1-\int_0^1 \rho$ and the length of $\hat \gamma$ is at most $-1$.

For the third case, observe that by Lemma \ref{caseslemma}, $F$ is the unique geodesic of $e$ that passes through $(1,a)$. In other words,

$$
e(0,0,y,t)+e(y,t,1,a)+e(1,a,0,1) < e(0,0,0,1), \forall t \in (0,a), \forall y \neq F(t)
$$
$$
e(0,0,1,a)+e(1,a,y,t)+e(y,t,0,1) < e(0,0,0,1), \forall t \in (a,1), \forall y \neq F(t).
$$

The continuous functions

$$
    \sup_{z \geq 1}e(0,0,y,t)+e(y,t,z,a)+e(z,a,0,1)
$$ 
$$
\sup_{z \geq 1}e(0,0,z,a)+e(z,a,y,t)+e(y,t,0,1)
$$
 attain their maximum values over the compact set 
 $$ \{ (y,t) : \min(|t|,|t-a|,|t-1|) \geq \frac{\delta}2, |y - F(t)| \geq \frac{\varepsilon}2, |y| \leq N \},
 $$ and since no geodesic of $e$ passes through both $(z,a)$ and $(y,t)$, these maxima are negative. Thus there exists $m > 0$ such that every curve $\hat \gamma$ passing through $(z,a)$ with $z \geq 1$ and $|\hat \gamma-F|_u \geq \varepsilon$ has length at most $-m$ in $e$. If $|\hat e - e| \leq \alpha$, then the length of $\hat \gamma$ in $\hat e$ is at most $-m+3\alpha$.
 
 We can now set $U$ to be the set of metrics that differ from $e$ by at most $\frac{m}5$ whenever $|x|,|y| \leq N$. Any geodesic $\hat \gamma$ with $\hat \gamma(a) \geq 1$ must fall into one of the above three cases, so the length of $\hat \gamma$ in any $\hat e \in U$ is at most $-\frac {2m}5$, while $\hat e(0,0,0,1) \geq -\frac m5$, so that $\hat \gamma$ is not a geodesic for $\hat e$.

\end{proof}

 \bibliographystyle{dcu}
 \bibliography{conjbib}

\end{document}